\documentclass[11pt, a4paper]{amsart}
\usepackage[utf8]{inputenc}
\usepackage[a4paper,hmargin=22mm]{geometry}

\usepackage{dsfont}
\usepackage{polynom}
\usepackage{empheq}
\usepackage{subfiles}
\usepackage{systeme}
\usepackage{shuffle}
\usepackage{ytableau}
\usepackage{CJK}
\usepackage{bm}
\usepackage{faktor}
\usepackage{ifthen}
\usepackage{mathrsfs}
\usepackage{amsmath}
\usepackage{mathtools}
\usepackage{amssymb}
\usepackage{amscd}
\usepackage{appendix}
\usepackage{amsthm}
\usepackage{amsfonts}
\usepackage[all]{xy}
\usepackage{array}
\usepackage{enumerate}
\usepackage{graphicx}
\usepackage{hyperref}
\usepackage{longtable}
\usepackage[OT2,T1]{fontenc}
\usepackage{verbatim}
\usepackage{tikz}
\usetikzlibrary{matrix}
\usepackage{url,setspace,multirow,tikz-cd}
\DeclareSymbolFont{cyrletters}{OT2}{wncyr}{m}{n}
\DeclareMathSymbol{\Sha}{\mathalpha}{cyrletters}{"58}
\usepackage{xcolor}
\usepackage{textcomp}
\usepackage{manyfoot}
\usepackage{booktabs}
\usepackage{algorithm}
\usepackage{algorithmicx}
\usepackage{algpseudocode}
\usepackage{listings}
\usepackage{shuffle}
\usepackage{ytableau}
\usepackage{mathdots}

\theoremstyle{plain}
\newtheorem{thm}{Theorem}[section]

\newtheorem{prop}[thm]{Proposition}%
\newtheorem{cor}[thm]{Corollary}

\theoremstyle{definition}
\newtheorem{eg}[thm]{Example}
\newtheorem{rem}[thm]{Remark}

\newtheorem{defn}[thm]{Definition}

\numberwithin{equation}{section}

\newcommand{\ZZ}{{\mathbb Z}}
\newcommand{\QQ}{{\mathbb Q}}

\newcommand{\RR}{{\mathbb R}}
\newcommand{\CC}{{\mathbb C}}

\newcommand{\NN}{{\mathbb N}}

\newcommand{\PP}{{\mathbb P}}

\newcommand{\mfk}{\mathfrak{m}}

\newcommand{\Bfk}{\mathfrak{B}}

\newcommand{\Hfk}{\mathfrak{H}}

\newcommand{\Xfk}{\mathfrak{X}}

\newcommand{\Hcal}{\mathcal{H}}

\newcommand{\Ocal}{\mathcal{O}}
\newcommand{\Pcal}{\mathcal{P}}

\newcommand{\Scal}{\mathcal{S}}

\newcommand{\Dcal}{\mathcal{D}}

\newcommand{\Bcal}{\mathcal{B}}

\newcommand{\Frob}{\operatorname{Frob}}

\begin{document}

\title[$p$MLFs and TMBNs]{$p$-adic multiple $L$-functions and twisted multiple Bernoulli numbers}

\author{Ku-Yu Fan}

\address{Graduate School of Mathematics, Nagoya University, Furo-cho, Chikusa-ku, Nagoya, 464-8602, Japan.}
\email{ku-yu.fan.d2@math.nagoya-u.ac.jp }

\date{\today}


\begin{abstract}
We compute the special values ($p$MLFVs) of the $p$-adic multiple $L$-functions introduced by Furusho, Komori, Matsumoto, and Tsumura at tuples of positive integers. Furusho and Jarossay show that the special values can be expressed as an infinite sum of cyclotomic multiple harmonic values (CMHVs) with coefficients given by cyclotomic multiple Bernoulli numbers (CMBNs). We provide an explicit formula for CMBNs in terms of twisted multiple Bernoulli numbers (TMBNs), which are special values of generalized Euler–Zagier–Lerch type complex multiple zeta functions at tuples of non-positive integers. As a result, we obtain that these $p$MLFVs can be expressed as infinite sums of CMHVs, with coefficients given by the special values of the complex functions at tuples of non-positive integers.
\end{abstract}



\maketitle

\tableofcontents

\section{Introduction}\label{section Introduction}
Furusho, Komori, Matsumoto, and Tsumura \cite{FKMTp-adic} constructed $p$-adic multiple $L$-functions ($p$MLFs, cf. Definition \ref{defn p-adic multiple L-function}) as multi-variable analogues of the Kubota-Leopoldt $p$-adic $L$-functions
$$L_p(s; \omega^k) \coloneqq \frac{1}{\langle c \rangle^{1-s}\omega^k(c) - 1} \int_{\ZZ_p^\times}\langle x \rangle^{-s} \omega^{k - 1}(x)d\widetilde{\mfk}_c(x).$$
In this paper, we study the special values of these $p$MLFs ($p$MLFVs) at tuples of positive integers. More precisely, for positive integers $n_1,\ldots,n_r$, we treat the family
\begin{equation}\label{eq pMLFVs}
  \left(p^{n_1+\cdots+n_r}L_{p,r}((n_i)_r;(\omega^{-n_i})_r;(1)_r;c)\right)_{p\in\Pcal_c}.
\end{equation}

Furusho and Jarossay \cite{CMBN} proved that the family in \eqref{eq pMLFVs} admits an explicit expansion (see Theorem \ref{thm CMBN}) as a convergent infinite series indexed by $\bm{l}=(l_i)_r\in\NN_0^r$, roots of unity $(\epsilon_i)_r\in(\mu_c\setminus\{1\})^r$, and combinatorial data $J\in E_r$ (cf.\ Definition \ref{defn J}). In their formula, the coefficients are given by cyclotomic multiple Bernoulli numbers (CMBNs, cf. Definition \ref{defn CMBN}), and the summands are expressed in terms of cyclotomic multiple harmonic values (CMHVs, cf. Definition \ref{defn cyclotomic multiple harmonic value}).

The purpose of this paper is to make the CMBN coefficients in this expansion more explicit. We employ the twisted Bernoulli numbers, which appear in the study of the desingularization of multivariable multiple zeta functions \cite{FKMTcomplex}, and which are defined as the coefficients of the exponential generating function
$$\frac{1}{1-\xi e^t}=\sum_{n=-1}^{\infty}\Bfk_n(\xi)\frac{t^n}{n!}, \qquad (\,(-1)! \coloneqq 1\,),$$
where $\xi$ is a root of unity. When $\xi=1$, they recover the classical Bernoulli numbers in the sense that $\Bfk_n(1) = -\frac{B_{n+1}}{n+1}$ for $n \geq 0$.

Their depth $r$ analogue is given by the twisted multiple Bernoulli numbers (TMBNs). For roots of unity $\xi_1,\ldots,\xi_r$ with $\xi_j\neq 1$ and parameters $\gamma_1,\ldots,\gamma_r \in \CC$, they are defined as the coefficients of the generating series
$$\prod_{j=1}^{r}\frac{1}{1-\xi_j \exp \left(\gamma_j \sum_{k=j}^{r}t_k \right)} = \sum_{n_1,\ldots,n_r \geq 0} \Bfk \left((n_j); (\xi_j); (\gamma_j)\right) \prod_{k=1}^{r} \frac{t_k^{n_k}}{n_k!}.$$
Furusho, Komori, Matsumoto, and Tsumura \cite{FKMTcomplex} further showed that these coefficients are given by special values of the generalized Euler-Zagier-Lerch type multiple zeta function
$$\zeta_r((s_j);(\xi_j);(\gamma_j)) = \sum_{m_1, \ldots, m_r> 0} \prod_{j=1}^{r}\xi_j^{m_j}(m_1\gamma_1+\cdots+m_j\gamma_j)^{-s_j},$$
by proving that $\zeta_r((s_j);(\xi_j);(\gamma_j))$ admits an analytic continuation to $\CC^r$ and that for $n_1,\ldots,n_r\in\NN_0$,
\begin{equation}\label{eq special value formula}
  \zeta_r((-n_j);(\xi_j);(\gamma_j))=(-1)^{r+n_1+\cdots+n_r}\Bfk\left((n_j);(\xi_j^{-1});(\gamma_j)\right).
\end{equation}

The main technical step of this paper is to connect the CMBNs appearing in the Furusho-Jarossay expansion with these TMBNs. In our first main result, Theorem \ref{thm explicit formula}, we prove an explicit formula expressing CMBNs in terms of TMBNs. The CMBNs appearing in the Furusho-Jarossay expansion
(see Theorem \ref{thm CMBN}) concern only a restricted family of the general CMBNs.  The explicit formula for this family is given in Corollary \ref{cor l = 0 formula}. Combining this formula with the special value formula \eqref{eq special value formula} above allows us to replace the CMBN coefficients in the Furusho-Jarossay expansion by special values of generalized Euler-Zagier-Lerch type multiple zeta functions.

Our second main result (Theorem \ref{thm Main thm}) of this paper is a reformulation of the expansion obtained by Furusho and Jarossay \cite{CMBN} (see Theorem \ref{thm CMBN}). We rewrite the family in \eqref{eq pMLFVs} as an infinite series indexed by the same data, in which the coefficients are expressed in terms of special values of the generalized Euler-Zagier-Lerch-type multiple zeta function at tuples of non-positive integers.
For example, in depth $1$, the expansion becomes
\begin{multline*}
  \left( p^{n}L_{p,1}(n; \omega^{-n}; 1; c)\right)_{p \in \Pcal_{c}} =  \\
  \sum_{l = 0}^{\infty} \sum_{\epsilon_1 \in \mu_{c} \setminus \{1\}} {-n \choose l} (-1)^{1 + l} \zeta_{1}\left(-l; \epsilon_1^{-1}; 1\right) \sum_{\delta \in \Delta(T_{1, (\emptyset, \{1\}, \emptyset)})}  \Hfk({\bf w}(l)_{\delta})^{\Frob^{-1}}
\end{multline*}
(cf. Corollary \ref{cor depth 1}).

This paper is organized as follows. In \S \ref{section Generalization of Bernoulli numbers}, we recall TMBNs and the CMBNs, and prove the explicit formula for CMBNs in terms of TMBNs. In particular, Corollary \ref{cor l = 0 formula} gives the coefficient formula used in the main theorem. In \S \ref{section Explicit formula for pMLFVs}, we recall the special value formula for the generalized Euler-Zagier-Lerch type multiple zeta function and the Furusho-Jarossay expansion of $p$MLFs. We then combine these results to prove Theorem \ref{thm Main thm}.

\section{Generalization of Bernoulli numbers}\label{section Generalization of Bernoulli numbers}
In this section, we recall two generalizations of Bernoulli numbers: TMBNs and the CMBNs. We then prove an explicit formula which describes CMBNs in terms of TMBNs in Theorem \ref{thm explicit formula}.
\begin{defn}[{\cite[p. 456]{Koblitz}}]\label{defn TBN}
  For any root of unity $\xi$, we define the {\it twisted Bernoulli numbers} $\Bfk_n(\xi)$ for $n \in \ZZ_{\geq -1}$ by
  $$\Hcal(t; \xi) = \frac{1}{1-\xi e^t} = \sum_{n = -1}^{\infty} \Bfk_n(\xi) \frac{t^n}{n!},$$
  where we formally let $(-1)! = 1$.
\end{defn}

\begin{eg}[{\cite[(1.3), (1.4)]{FKMTcomplex}}]\label{eg TBN}
  In the case $\xi = 1$, we have
  $$\Bfk_{-1}(1) = -1, \  \Bfk_{n}(1) = -\frac{B_{n+1}}{n+1}\ (n \in \NN_{0}),$$
  where $B_{n}$ denotes the $n$-th Bernoulli number with the convention $B_1 = -\frac{1}{2}$, defined by
  $$\frac{t}{e^t - 1} = \sum_{n = 0}^{\infty} B_{n} \frac{t^n}{n!}.$$
  In the case $\xi \neq 1$, we have
  $$\Bfk_{0}(\xi) = \frac{1}{1-\xi}, \ \Bfk_{1}(\xi) = \frac{\xi}{(1-\xi)^2}, \ \Bfk_{2}(\xi) = \frac{\xi(\xi+1)}{(1-\xi)^3},$$
  $$\Bfk_{3}(\xi) = \frac{\xi(\xi^2+4\xi+1)}{(1-\xi)^4}, \ \Bfk_{4}(\xi) = \frac{\xi(\xi^3+11\xi^2+11\xi+1)}{(1-\xi)^5}, \ldots$$
\end{eg}

\begin{defn}[{\cite[Definition 1.5]{FKMTcomplex}}]\label{defn TMBN}
  Let $r \in \NN$, $\gamma_{1}, \ldots, \gamma_{r} \in \CC$ and let $\xi_1, \ldots, \xi_r \in \CC \setminus \{1\}$ be roots of unity. Set
  $$\Hcal_{r}((t_j); (\xi_j); (\gamma_j)) \coloneqq \prod_{j = 1}^{r}\Hcal \left(\gamma_j \left( \sum_{k = j}^{r} t_k \right); \xi_j \right) = \prod_{j = 1}^{r}\frac{1}{1-\xi_j \mathrm{exp}(\gamma_j \sum_{k = j}^{r} t_k)}$$
  and define the {\it twisted multiple Bernoulli numbers} (TMBNs) $\Bfk((n_j); (\xi_j); (\gamma_j))$ for $(n_j) \in \ZZ_{\geq 0}^r$ by
  $$\Hcal_{r}((t_j); (\xi_j); (\gamma_j)) = \sum_{n_1 = 0}^{\infty} \cdots \sum_{n_r = 0}^{\infty} \Bfk((n_j); (\xi_j); (\gamma_j)) \frac{t_{1}^{n_{1}}}{n_{1}!} \cdots \frac{t_{r}^{n_{r}}}{n_{r}!}.$$
\end{defn}

\begin{rem}\label{rem TMBN}
  In the case $r = 1$, we have $\Bfk(n; \xi; 1) = \Bfk_{n}(\xi)$. In the general case, we use the notation $\Bfk_{(n_j)_r}((\xi_j)_r) \coloneqq \Bfk((n_j)_r; (\xi_j)_r; (1)_r)$ when $\gamma_1 = \cdots = \gamma_r = 1$.
\end{rem}

\begin{defn}
  Let $c \in \ZZ_{p}^{\times}\cap\NN$ with $c\geq 2$. Put $(l_{i})_{r}=(l_{1},\ldots,l_{r})\in \NN_{0}^{r}$ and $(\epsilon_i)_r=(\epsilon_{1},\ldots,\epsilon_{r})\in\mu_c^r$. For $h \in \NN$, and $(\kappa_{i})_{r-1}=(\kappa_{1},\dots,\kappa_{r-1})\in \NN_{0}^{r-1}$, we define the {\it cyclotomic multiple harmonic sums with modified $(\kappa_{i})_{r-1}$-steps} by
  \begin{equation}\label{eq Scal}
    \Scal_{(\kappa_{i})_{r-1},h}((l_{i})_{r};(\epsilon_{i})_{r}) = \sum_{\substack{(u_{1},\ldots,u_{r}) \in \NN^r_{0},\ u_{1} < h \\ \forall i \geq 2,\ u_{i-1}+\kappa_{i-1} h < u_{i} \\ \forall i \geq 2,\ u_{i}<u_{i-1}+(\kappa_{i-1}+1)h}} \left( \frac{\epsilon_{2}}{\epsilon_{1}}\right)^{u_{1}} \cdots \left(\frac{1}{\epsilon_{r}}\right)^{u_{r}} u_{1}^{l_{1}} \cdots u_{r}^{l_{r}}\in \QQ(\mu_{c}),
  \end{equation}
  where we formally let $0^0 = 1$.
\end{defn}

\begin{defn}[{\cite[Lemma 6]{CMBN}}]\label{defn CMBN}
  Let $c\in \ZZ^{\times}_{p} \cap \NN, c\geq 2$. For any $(l_{i})_{r} \in \NN_{0}^{r}$, $(\epsilon_{i})_{r}\in \mu_{c}^{r}$, and $(\kappa_{i})_{r-1} \in \NN_{0}^{r-1}$, we define the {\it cyclotomic multiple Bernoulli numbers} (CMBNs) with modified $(\kappa_{i})_{r-1}$-steps $\Bcal_{l,\xi}^{((l_{i})_{r};(\epsilon_{i})_{r};(\kappa_{i})_{r-1})}$ for $l \in \{0, \ldots, l_{1}+\cdots +l_{r}+r\}$ and $\xi \in \mu_{c}$ to be the elements in $\QQ(\mu_{c})$ such that, for all $h \in \NN$, we have
  \begin{equation*}
    \Scal_{(\kappa_{i})_{r-1},h}((l_{i})_{r};(\epsilon_{i})_{r}) = \sum_{\substack{0 \leq l \leq  l_{1}+\cdots+l_{r}+r \\ \xi \in \mu_{c}}}\Bcal_{l,\xi}^{((l_{i})_{r};(\epsilon_{i})_{r};(\kappa_{i})_{r-1})} h^{l} \xi^{h}.
  \end{equation*}
\end{defn}

\begin{thm}\label{thm explicit formula}
  Let $c \in \ZZ_{p}^{\times}\cap\NN$, $c\geq 2$. Put $(l_{i})_{r}=(l_{1},\ldots,l_{r})\in \NN_{0}^{r}$ and $(\epsilon_{1},\ldots,\epsilon_{r})\in\mu_c^r$. Let $(\kappa_{i})_{r-1}=(\kappa_{1},\dots,\kappa_{r-1})\in \NN_{0}^{r-1}$ and set $\kappa_{0} = 0$. Then, the following formula
  \begin{align*}
    & \Bcal_{l,\xi}^{((l_{i})_{r};(\epsilon_{i})_{r};(\kappa_{i})_{r-1})} \\
    = & \sum_{\mathbf{e} = (e_i)_r \in \{0\}\times \{0, 1\}^{r-1}} \sum_{\substack{\forall 1 \leq i \leq j \leq r \ m_{\mathbf{e},i,j}\in \NN_0 \\ m_{\mathbf{e},1,1} \leq l_1 \\ \cdots \\ m_{\mathbf{e},1,r} + \cdots + m_{\mathbf{e},r,r} \leq l_r}} \sum_{i_1 = 0}^{1-e_1} \cdots \sum_{i_r = 0}^{1-e_r} (-1)^{\sum_{k = 1}^{r} (e_k + i_k)} \prod_{k = 1}^{r} (\kappa_{k-1} + i_k)^{\sum_{j = k}^{r}m _{\mathbf{e}, k, j}} \\
    & \prod_{k = 1}^{r} \binom{l_{k}}{m_{\mathbf{e},1,k}, \ldots, m_{\mathbf{e},k,k}, l_{k} - \left( \sum_{i = 1}^{k} m_{\mathbf{e},i,k} \right)} \delta_{\xi, \prod_{k = 1}^{r} \epsilon_{k}^{-(\kappa_{k-1} + i_k)}} \delta_{l, \sum_{k = 1}^{r} \sum_{j = k}^{r}m _{\mathbf{e}, k, j}} \\
    & \Bfk_{\left(l_{1} - m_{\mathbf{e},1,1} \circ_{e_2} \cdots \circ_{e_r} l_{r} - \left( \sum_{i = 1}^{r} m_{\mathbf{e},i,r} \right)\right)} (\epsilon_1^{-1} (1 - e_1) \circ_{e_2} \cdots \circ_{e_r} \epsilon_r^{-1} (1 - e_r))
  \end{align*}
  holds for every $l, \xi$, where $\circ_0 = ,$ and $\circ_1 = +$.
\end{thm}

\begin{proof}
  Consider the following generating function, whose coefficients are $\Scal_{(\kappa_{i})_{r-1},h}((l_{i})_{r};(\epsilon_{i})_{r})$
  $$F(t_1, \ldots, t_r) = \sum_{l_{1}, \ldots, l_{r} = 0}^{\infty} \Scal_{(\kappa_{i})_{r-1},h}((l_{i})_{r};(\epsilon_{i})_{r}) \prod_{k = 1}^{r} \frac{t_{k}^{l_{k}}}{l_{k}!}$$
  For convenience, we let $\epsilon_k^{-1} = \xi_k$, $\epsilon_{r+1}=1$, $\kappa_{0} = 0$. Then we have
  $$F(t_1, \ldots, t_r) = \sum_{l_{1}, \ldots, l_{r} = 0}^{\infty} \left[ \sum_{\substack{(u_{1},\ldots,u_{r}) \in \NN_{0}^r,\ u_{1} < h \\ \forall i \geq 2,\ u_{i-1}+\kappa_{i-1} h<u_{i} \\ \forall i \geq 2,\ u_{i}<u_{i-1}+(\kappa_{i-1}+1)h}} \prod_{k = 1}^{r} \left( \frac{\epsilon_{k+1}}{\epsilon_{k}}\right)^{u_{k}} u_{k}^{l_{k}} \right] \prod_{k = 1}^{r} \frac{t_{k}^{l_{k}}}{l_{k}!}.$$
  We swapped the sum and the product to get
  $$F(t_1, \ldots, t_r) = \sum_{\substack{(u_{1},\ldots,u_{r}) \in \NN_{0}^r,\ u_{1} < h \\ \forall i \geq 2,\ u_{i-1}+\kappa_{i-1} h<u_{i} \\ \forall i \geq 2,\ u_{i}<u_{i-1}+(\kappa_{i-1}+1)h}} \prod_{k = 1}^{r} \left[ \left( \frac{\epsilon_{k+1}}{\epsilon_{k}}\right)^{u_{k}} \sum_{l_k = 0}^{\infty} \frac{(u_{k}t_{k})^{l_{k}}}{l_{k}!} \right].$$
  We use the expansion of the exponential function to obtain
  $$F(t_1, \ldots, t_r) = \sum_{\substack{(u_{1},\ldots,u_{r}) \in \NN_{0}^r,\ u_{1} < h \\ \forall i \geq 2,\ u_{i-1}+\kappa_{i-1} h<u_{i} \\ \forall i \geq 2,\ u_{i}<u_{i-1}+(\kappa_{i-1}+1)h}} \prod_{k = 1}^{r} \left( \frac{\epsilon_{k+1}}{\epsilon_{k}} \mathrm{exp}(t_{k}) \right)^{u_{k}}.$$
  We expand the product of the $r$-th term to obtain
  \begin{multline*}
    F(t_1, \ldots, t_r) =  \\
    \sum_{\substack{(u_{1},\ldots,u_{r-1}) \in \NN_{0}^{r-1},\ u_{1} < h \\ \forall i \geq 2,\ u_{i-1}+\kappa_{i-1} h<u_{i} \\ \forall i \geq 2,\ u_{i}<u_{i-1}+(\kappa_{i-1}+1)h}} \prod_{k = 1}^{r-1} \left( \frac{\epsilon_{k+1}}{\epsilon_{k}} \mathrm{exp}(t_{k}) \right)^{u_{k}} \\
    \times \left( \sum_{u_r = u_{r-1}+\kappa_{r-1} h}^{u_{r-1}+(\kappa_{r-1}+1)h-1} \left( \xi_r \mathrm{exp}(t_{r}) \right)^{u_{r}} - \left( \xi_r \mathrm{exp}(t_{r}) \right)^{u_{r-1}+\kappa_{r-1} h} \right).
  \end{multline*}
  For convenience, we let
  $$E_{k, e} = \left(\xi_{k}\mathrm{exp}\left(\sum_{j = k}^{r} t_{j}\right)\right)^{(\kappa_{k-1}+e)h}.$$
  Then we have
  \begin{multline*}
    F(t_1, \ldots, t_r) = \\
    \sum_{\substack{(u_{1},\ldots,u_{r-1}) \in \NN_{0}^{r-1},\ u_{1} < h \\ \forall i \geq 2,\ u_{i-1}+\kappa_{i-1} h<u_{i} \\ \forall i \geq 2,\ u_{i}<u_{i-1}+(\kappa_{i-1}+1)h}} \prod_{k = 1}^{r-1} \left( \frac{\epsilon_{k+1}}{\epsilon_{k}} \mathrm{exp}(t_{k}) \right)^{u_{k}} \\
    \times \left( \frac{\left(\xi_{r}\mathrm{exp}(t_{r})\right)^{u_{r-1}}(E_{r, 0} - E_{r, 1})}{1 - \xi_{r}\mathrm{exp}(t_{r})}  - \left( \xi_r \mathrm{exp}(t_{r}) \right)^{u_{r-1}}E_{r, 0} \right).
  \end{multline*}
  We combine the $(r-1)$-th terms of the product to get
  \begin{multline*}
    F(t_1, \ldots, t_r) = \\
    \sum_{\substack{(u_{1},\ldots,u_{r-1}) \in \NN_{0}^{r-1},\ u_{1} < h \\ \forall i \geq 2,\ u_{i-1}+\kappa_{i-1} h<u_{i} \\ \forall i \geq 2,\ u_{i}<u_{i-1}+(\kappa_{i-1}+1)h}} \prod_{k = 1}^{r-2} \left( \frac{\epsilon_{k+1}}{\epsilon_{k}} \mathrm{exp}(t_{k}) \right)^{u_{k}} \\
    \times \left( \xi_{r-1} \mathrm{exp}(t_{r-1} + t_r) \right)^{u_{r-1}} \left( \frac{E_{r, 0} - E_{r, 1}}{1 - \xi_{r}\mathrm{exp}(t_{r})}  - E_{r, 0} \right).
  \end{multline*}
  We repeat this process until we get the first term to obtain
  $$F(t_1, \ldots, t_r) = \sum_{u_1 = 0}^{h - 1} \left(\xi_{1}\mathrm{exp}\left(\sum_{j = 1}^{r} t_{j}\right)\right)^{u_1} \prod_{k = 2}^{r} \left( \frac{E_{k, 0} - E_{k, 1}}{1 - \xi_{k}\mathrm{exp}\left(\sum_{j = k}^{r} t_{j}\right)} - E_{k, 0} \right).$$
  We also expand the first item to obtain
  $$F(t_1, \ldots, t_r) = \frac{E_{1, 0} - E_{1, 1}}{1 - \xi_{1}\mathrm{exp}\left(\sum_{j = 1}^{r} t_{j}\right)} \prod_{k = 2}^{r} \left( \frac{E_{k, 0} - E_{k, 1}}{1 - \xi_{k}\mathrm{exp}\left(\sum_{j = k}^{r} t_{j}\right)} - E_{k, 0} \right).$$
  We use index $(e_i)_r$ to rewrite the product as a sum to obtain
  $$F(t_1, \ldots, t_r) = \sum_{(e_i)_r \in \{0\}\times \{0, 1\}^{r-1}} \prod_{k = 1}^r \left(\frac{E_{k, 0} - E_{k, 1}}{1 - \xi_{k}\mathrm{exp}\left(\sum_{j = k}^{r} t_{j}\right)}\right)^{1-e_k} (- E_{k, 0})^{e_k}.$$
  For $(e_i)_r \in \{0\}\times \{0, 1\}^{r-1}$, we let $\circ_0 = ,$ and $\circ_1 = +$. Then
  \begin{align*}
    & \prod_{k = 1}^r \left(\frac{1}{1 - \xi_{k}\mathrm{exp}\left(\sum_{j = k}^{r} t_{j}\right)}\right)^{1-e_k} \\
    = & \Hcal_{r-\sum_{k = 1}^{r}e_k}((t_1 \circ_{e_2} \cdots \circ_{e_r} t_r); (\xi_1 (1 - e_1) \circ_{e_2} \cdots \circ_{e_r} \xi_r (1 - e_r)); ((1 - e_1) \circ_{e_2} \cdots \circ_{e_r} (1 - e_r))) \\
    = & \sum_{n_1, \ldots, n_r = 0}^{\infty} \Bfk_{(n_1 \circ_{e_2} \cdots \circ_{e_r} n_r)}(\xi_1 (1 - e_1) \circ_{e_2} \cdots \circ_{e_r} \xi_r (1 - e_r)) \prod_{k = 1}^{r} \frac{t_{k}^{n_{k}}}{n_{k}!}.
  \end{align*}
  We rewrite $E_{k, e}$ as the series. Then
  \begin{align*}
    - E_{k, 0} = & - \xi_{k}^{\kappa_{k-1} h} \sum_{n = 0}^{\infty}\frac{\left((\kappa_{k-1} h)\sum_{j = k}^{r} t_{j}\right)^n}{n!} \\
    = & \sum_{n = 0}^{\infty} (- \xi_{k}^{\kappa_{k-1} h})(\kappa_{k-1} h)^n \frac{\left(\sum_{j = k}^{r} t_{j}\right)^n}{n!} \\
    = & \sum_{n_k, \ldots, n_r = 0}^{\infty} (- \xi_{k}^{\kappa_{k-1} h})(\kappa_{k-1} h)^{\sum_{j = k}^{r} n_j} \prod_{j = k}^{r}\frac{t_j^{n_j}}{{n_j}!}
  \end{align*}
  and
  \begin{align*}
    E_{k, 0} - E_{k, 1} = & \sum_{n = 0}^{\infty}\frac{\left(\xi_{k}^{\kappa_{k-1} h} (\kappa_{k-1} h)^n - \xi_{k}^{(\kappa_{k-1} + 1) h} \left((\kappa_{k-1} + 1) h\right)^n\right)\left(\sum_{j = k}^{r} t_{j}\right)^n}{n!} \\
    = & \sum_{n_k, \ldots, n_r = 0}^{\infty} \left(\xi_{k}^{\kappa_{k-1} h} (\kappa_{k-1} h)^{\sum_{j = k}^{r} n_j} - \xi_{k}^{(\kappa_{k-1} + 1) h} \left((\kappa_{k-1} + 1) h\right)^{\sum_{j = k}^{r} n_j}\right) \prod_{j = k}^{r}\frac{t_j^{n_j}}{{n_j}!}.
  \end{align*}
  To obtain $\Scal_{(\kappa_{i})_{r-1},h}((l_{i})_{r};(\epsilon_{i})_{r})$, we consider the coefficient of $\prod_{k = 1}^{r} \frac{t_{k}^{l_{k}}}{l_{k}!}$
  \begin{align*}
    & \Scal_{(\kappa_{i})_{r-1},h}((l_{i})_{r};(\epsilon_{i})_{r}) \prod_{k = 1}^{r} \frac{t_{k}^{l_{k}}}{l_{k}!} \\
    = & \sum_{\mathbf{e} = (e_i)_r \in \{0\}\times \{0, 1\}^{r-1}} \sum_{\substack{m_{\mathbf{e},i,j}\in \NN_0 \\ m_{\mathbf{e},1,1} \leq l_1 \\ \cdots \\ m_{\mathbf{e},1,r} + \cdots + m_{\mathbf{e},r,r} \leq l_r}} \\
    & \Bfk_{\left(l_{1} - m_{\mathbf{e},1,1} \circ_{e_2} \cdots \circ_{e_r} l_{r} - \left( \sum_{i = 1}^{r} m_{\mathbf{e},i,r} \right)\right)} (\xi_1 (1 - e_1) \circ_{e_2} \cdots \circ_{e_r} \xi_r (1 - e_r)) \\
    & \prod_{k = 1}^{r} \frac{t_{k}^{l_{k} - (m_{\mathbf{e},1,k} + \cdots + m_{\mathbf{e},k,k})}}{(l_{k} - (m_{\mathbf{e},1,k} + \cdots + m_{\mathbf{e},k,k}))!} \\
    & \prod_{k = 1}^{r} \left( \xi_{k}^{\kappa_{k-1} h} (\kappa_{k - 1}h)^{\sum_{j = k}^{r} m_{\mathbf{e},k, j}} - \xi_{k}^{(\kappa_{k-1} + 1)h} ((\kappa_{k - 1} + 1)h)^{\sum_{j = k}^{r} m_{\mathbf{e},k, j}} \right)^{1-e_k} \\
    & \left((- \xi_{k}^{\kappa_{k-1} h})(\kappa_{k-1} h)^{\sum_{j = k}^{r} m_{\mathbf{e},k, j}}\right)^{e_k} \prod_{j = k}^{r} \frac{t_{j}^{m_{\mathbf{e},k,j}}}{m_{\mathbf{e},k, j}!}.
  \end{align*}
  Then we obtain $\Scal_{(\kappa_{i})_{r-1},h}((l_{i})_{r};(\epsilon_{i})_{r})$ by comparing the coefficient
  \begin{align*}
    & \Scal_{(\kappa_{i})_{r-1},h}((l_{i})_{r};(\epsilon_{i})_{r}) \prod_{k = 1}^{r} \frac{t_{k}^{l_{k}}}{l_{k}!} \\
    = & \sum_{\mathbf{e} = (e_i)_r \in \{0\}\times \{0, 1\}^{r-1}} \sum_{\substack{m_{\mathbf{e},i,j}\in \NN_0 \\ m_{\mathbf{e},1,1} \leq l_1 \\ \cdots \\ m_{\mathbf{e},1,r} + \cdots + m_{\mathbf{e},r,r} \leq l_r}} \\
    & \Bfk_{\left(l_{1} - m_{\mathbf{e},1,1} \circ_{e_2} \cdots \circ_{e_r} l_{r} - \left( \sum_{i = 1}^{r} m_{\mathbf{e},i,r} \right)\right)} (\xi_1 (1 - e_1) \circ_{e_2} \cdots \circ_{e_r} \xi_r (1 - e_r)) \\
    & \prod_{k = 1}^{r} \left( \xi_{k}^{\kappa_{k-1} h} (\kappa_{k - 1}h)^{\sum_{j = k}^{r} m_{\mathbf{e},k, j}} - \xi_{k}^{(\kappa_{k-1} + 1)h} ((\kappa_{k - 1} + 1)h)^{\sum_{j = k}^{r} m_{\mathbf{e},k, j}} \right)^{1-e_k} \\
    & \left((- \xi_{k}^{\kappa_{k-1} h})(\kappa_{k-1} h)^{\sum_{j = k}^{r} m_{\mathbf{e},k, j}}\right)^{e_k} \prod_{k = 1}^{r} \binom{l_{k}}{m_{\mathbf{e},1,k}, \ldots, m_{\mathbf{e},k,k}, l_{k} - (m_{\mathbf{e},1,k} + \cdots + m_{\mathbf{e},k,k})} \frac{t_{k}^{l_{k}}}{l_{k}!}.
  \end{align*}
  For $\mathbf{e} = (e_i)_r \in \{0\}\times \{0, 1\}^{r-1}$, we let $m_{\mathbf{e}, k} = \sum_{j = k}^{r}m _{\mathbf{e}, k, j}$, and we calculate
  \begin{align*}
    & \prod_{k = 1}^{r} \left( \xi_{k}^{\kappa_{k-1} h} (\kappa_{k - 1}h)^{m_{\mathbf{e}, k}} - \xi_{k}^{(\kappa_{k-1} + 1)h} ((\kappa_{k - 1} + 1)h)^{m_{\mathbf{e}, k}} \right)^{1-e_k} \left((- \xi_{k}^{\kappa_{k-1} h})(\kappa_{k-1} h)^{m_{\mathbf{e}, k}}\right)^{e_k} \\
    = & \sum_{i_1 = 0}^{1-e_1} \cdots \sum_{i_r = 0}^{1-e_r} (-1)^{\sum_{k = 1}^{r} (e_k + i_k)} \prod_{k = 1}^{r} \xi_{k}^{(\kappa_{k-1} + i_k) h} ((\kappa_{k-1} + i_k) h)^{m_{\mathbf{e}, k}}.
  \end{align*}
  By comparing the coefficient, we obtain
  \begin{align*}
    & \Bcal_{l,\xi}^{((l_{i})_{r};(\epsilon_{i})_{r};(\kappa_{i})_{r-1})} \\
    = & \sum_{\mathbf{e} = (e_i)_r \in \{0\}\times \{0, 1\}^{r-1}} \sum_{\substack{m_{\mathbf{e},i,j}\in \NN_0 \\ m_{\mathbf{e},1,1} \leq l_1 \\ \cdots \\ m_{\mathbf{e},1,r} + \cdots + m_{\mathbf{e},r,r} \leq l_r}} \sum_{i_1 = 0}^{1-e_1} \cdots \sum_{i_r = 0}^{1-e_r} (-1)^{\sum_{k = 1}^{r} (e_k + i_k)} \prod_{k = 1}^{r} (\kappa_{k-1} + i_k)^{m_{\mathbf{e}, k}}\\
    & \Bfk_{\left(l_{1} - m_{\mathbf{e},1,1} \circ_{e_2} \cdots \circ_{e_r} l_{r} - \left( \sum_{i = 1}^{r} m_{\mathbf{e},i,r} \right)\right)} (\xi_1 (1 - e_1) \circ_{e_2} \cdots \circ_{e_r} \xi_r (1 - e_r)) \\
    & \prod_{k = 1}^{r} \binom{l_{k}}{m_{\mathbf{e},1,k}, \ldots, m_{\mathbf{e},k,k}, l_{k} - (m_{\mathbf{e},1,k} + \cdots + m_{\mathbf{e},k,k})} \delta_{\xi, \prod_{k = 1}^{r} \xi_{k}^{(\kappa_{k-1} + i_k)}} \delta_{l, \sum_{k = 1}^{r} m_{\mathbf{e}, k}}.
  \end{align*}
  This proves the theorem.
\end{proof}

The following corollary is the explicit formula for $r = 1$.

\begin{cor}\label{cor r = 1 formula}
  Let $c \in \ZZ_{p}^{\times}\cap\NN$ with $c\geq 2$. Let $l_{1} \in \NN_{0}$ be a non-negative integer and $\epsilon_1 \in \mu_{c}$. Then, the following formula
  $$\Bcal_{l,\xi}^{(l_{1};\epsilon_1;\emptyset)} = \Bfk_{l_{1}}(\epsilon_1^{-1}) \delta_{\xi,1} \delta_{l, 0} - \sum_{m = 0}^{l_{1}} \binom{l_{1}}{m} \Bfk_{l_{1} - m}(\epsilon_1^{-1}) \delta_{\xi, \epsilon_1^{-1}} \delta_{l, m}$$
  holds for every $l, \xi$.
\end{cor}

\begin{proof}
  Substituting $r = 1$ into Theorem \ref{thm explicit formula}, we get
  \begin{align*}
    & \Bcal_{l,\xi}^{(l_{1};\epsilon_{1};\emptyset)} \\
    = & \sum_{e_1 = 0} \sum_{m_{e_1,1,1} = 0}^{l_1} \sum_{i_1 = 0}^{1-e_1} (-1)^{e_1 + i_1} (i_1)^{m _{e_1, 1, 1}} \Bfk_{l_{1} - m_{e_1,1,1}} (\epsilon_1^{-1} (1 - e_1)) \binom{l_{1}}{m_{e_1,1,1}} \delta_{\xi, \epsilon_1^{-i_1}} \delta_{l, m _{e_1, 1, 1}} \\
    = & \sum_{m_{0,1,1} = 0}^{l_1} \sum_{i_1 = 0}^{1} (-1)^{i_1} (i_1)^{m _{0, 1, 1}} \Bfk_{l_{1} - m_{0,1,1}} (\epsilon_1^{-1}) \binom{l_{1}}{m_{0,1,1}} \delta_{\xi, \epsilon_1^{-i_1}} \delta_{l, m _{0, 1, 1}} \\
    = & \sum_{m_{0,1,1} = 0}^{l_1} 0^{m _{0, 1, 1}} \Bfk_{l_{1} - m_{0,1,1}} (\epsilon_1^{-1}) \binom{l_{1}}{m_{0,1,1}} \delta_{\xi, 1} \delta_{l, m _{0, 1, 1}} - \sum_{m_{0,1,1} = 0}^{l_1} \Bfk_{l_{1} - m_{0,1,1}} (\epsilon_1^{-1}) \binom{l_{1}}{m_{0,1,1}} \delta_{\xi, \epsilon_1^{-1}} \delta_{l, m _{0, 1, 1}}.
  \end{align*}
  This proves the corollary.
\end{proof}

The following corollary is the explicit formula for $l = 0$.

\begin{cor}\label{cor l = 0 formula}
  Let $c \in \ZZ_{p}^{\times}\cap\NN$, $c\geq 2$. Put $(l_{i})_{r}=(l_{1},\ldots,l_{r})\in \NN_{0}^{r}$ and $(\epsilon_{1},\ldots,\epsilon_{r})\in\mu_c^r$. For $(\kappa_{i})_{r-1}=(\kappa_{1},\dots,\kappa_{r-1})\in \NN_{0}^{r-1}$, set $\kappa_{0} = 0$. Then, the following formula
  \begin{align*}
    \Bcal_{0,\xi}^{((l_{i})_{r};(\epsilon_{i})_{r};(\kappa_{i})_{r-1})} = & \sum_{(e_i)_r \in \{0\}\times \{0, 1\}^{r-1}} \sum_{i_1 = 0}^{1-e_1} \cdots \sum_{i_r = 0}^{1-e_r} (-1)^{\sum_{k = 1}^{r} (e_k + i_k)}\\
    & \Bfk_{\left(l_{1} \circ_{e_2} \cdots \circ_{e_r} l_{r}\right)} (\epsilon_1^{-1} (1 - e_1) \circ_{e_2} \cdots \circ_{e_r} \epsilon_r^{-1} (1 - e_r)) \delta_{\xi, \prod_{k = 1}^{r} \epsilon_{k}^{-(\kappa_{k-1} + i_k)}}
  \end{align*}
  holds for every $\xi$.
\end{cor}

\begin{proof}
  Substituting $l = 0$ into Theorem \ref{thm explicit formula}, we get
  \begin{align*}
    & \Bcal_{0,\xi}^{((l_{i})_{r};(\epsilon_{i})_{r};(\kappa_{i})_{r-1})} \\
    = & \sum_{\mathbf{e} = (e_i)_r \in \{0\}\times \{0, 1\}^{r-1}} \sum_{\substack{\forall 1 \leq i \leq j \leq r \ m_{\mathbf{e},i,j}\in \NN_0 \\ m_{\mathbf{e},1,1} \leq l_1 \\ \cdots \\ m_{\mathbf{e},1,r} + \cdots + m_{\mathbf{e},r,r} \leq l_r}} \sum_{i_1 = 0}^{1-e_1} \cdots \sum_{i_r = 0}^{1-e_r} (-1)^{\sum_{k = 1}^{r} (e_k + i_k)} \prod_{k = 1}^{r} (\kappa_{k-1} + i_k)^{\sum_{j = k}^{r}m _{\mathbf{e}, k, j}}\\
    & \Bfk_{\left(l_{1} - m_{\mathbf{e},1,1} \circ_{e_2} \cdots \circ_{e_r} l_{r} - \left( \sum_{i = 1}^{r} m_{\mathbf{e},i,r} \right)\right)} (\epsilon_1^{-1} (1 - e_1) \circ_{e_2} \cdots \circ_{e_r} \epsilon_r^{-1} (1 - e_r)) \\
    & \prod_{k = 1}^{r} \binom{l_{k}}{m_{\mathbf{e},1,k}, \ldots, m_{\mathbf{e},k,k}, l_{k} - (m_{\mathbf{e},1,k} + \cdots + m_{\mathbf{e},k,k})} \delta_{\xi, \prod_{k = 1}^{r} \epsilon_{k}^{-(\kappa_{k-1} + i_k)}} \delta_{0, \sum_{k = 1}^{r} \sum_{j = k}^{r}m _{\mathbf{e}, k, j}} \\
    = & \sum_{(e_i)_r \in \{0\}\times \{0, 1\}^{r-1}} \sum_{i_1 = 0}^{1-e_1} \cdots \sum_{i_r = 0}^{1-e_r} (-1)^{\sum_{k = 1}^{r} (e_k + i_k)}\\
    & \Bfk_{\left(l_{1} \circ_{e_2} \cdots \circ_{e_r} l_{r}\right)} (\epsilon_1^{-1} (1 - e_1) \circ_{e_2} \cdots \circ_{e_r} \epsilon_r^{-1} (1 - e_r)) \delta_{\xi, \prod_{k = 1}^{r} \epsilon_{k}^{-(\kappa_{k-1} + i_k)}}.
  \end{align*}
  This proves the corollary.
\end{proof}

\section{Explicit formula for pMLFVs}\label{section Explicit formula for pMLFVs}
In this section, we recall two results and then prove Theorem \ref{thm Main thm}. The first is a theorem from \cite{FKMTcomplex}, which identifies TMBNs with special values of the multiple zeta-function of the generalized Euler-Zagier-Lerch type. The second is a theorem from \cite{CMBN}, which gives an explicit expansion of a certain family of $p$MLFs as an infinite series indexed by combinatorial data, with coefficients given by CMBNs. We then prove a reformulation of this expansion in which the coefficients are expressed in terms of these special values in Theorem \ref{thm Main thm}.

\begin{defn}[\cite{FKMTcomplex}]
  Let $\xi_1,\ldots,\xi_r\in \CC$ be roots of unity. For $\gamma_1,\ldots,\gamma_r\in \CC$ with $\Re \gamma_j >0$ ($1 \leq j \leq r$), where $\Re$ denotes the real part, the {\it multiple zeta-function of the generalized Euler-Zagier-Lerch type} is defined by
  \begin{equation}
    \zeta_r((s_j); (\xi_j); (\gamma_j)) \coloneqq \sum_{m_1=1}^\infty \cdots \sum_{m_r=1}^\infty \prod_{j=1}^{r} \xi_j^{m_j} (m_1\gamma_1+\cdots+m_j\gamma_j)^{-s_j},
  \end{equation}
  which is absolutely convergent in the region
  $$\Dcal_r = \{(s_1, \ldots, s_r)\in \CC^r \big| \Re (s_{r-k+1} + \cdots + s_r) > k \ (1 \leq k \leq r)\}.$$
\end{defn}

\begin{thm}[{\cite[Theorem 2.1]{FKMTcomplex}}]\label{thm TMBN}
  Let $\xi_1,\ldots,\xi_r\in \CC$ be roots of unity and $\gamma_1,\ldots,\gamma_r\in \CC$ with $\Re \gamma_j >0 \ (1 \leq  j \leq  r)$. Assume that
  $$\xi_j \neq 1 \text{ for all } j \ (1 \leq  j \leq  r).$$
  Then, with the above notation, $\zeta_r((s_j);(\xi_j);(\gamma_j))$ can be analytically continued to $\CC^r$ as an entire function in $(s_j)$. For $n_1,\ldots,n_r\in \NN_0$,
  $$\zeta_r((-n_j); (\xi_j); (\gamma_j))=(-1)^{r + n_1 + \cdots + n_r}\Bfk((n_j); (\xi_j^{-1}); (\gamma_j)).$$
\end{thm}

\begin{rem}
  We use the notation $\zeta_r((s_j);(\xi_j)) \coloneqq \zeta_r((s_j);(\xi_j);(1, \ldots, 1))$ when $\gamma_j = 1$ for all $j$.
\end{rem}

Fix a prime number $p$. We put $\Ocal_{\CC_{p}}$ to be the ring of integers of $\CC_{p}$. Let $\omega: \Ocal_{\CC_{p}}^{\times}\to  \Ocal_{\CC_{p}}^{\times}$ be the Teichmüller character, and let $\langle x \rangle = \frac{x}{\omega(x)}$ for $x \in \Ocal_{\CC_{p}}^{\times}$. Set 
$$\int_{\ZZ_{p}} f(x) d \mfk_{z}(x) = \lim_{N \rightarrow \infty} \sum_{a=0}^{p^{N}-1} f(a) \mfk_{z}(a+p^{N}\ZZ_{p}),$$
where $\mfk_{z}$ is the measure defined by
$$\mfk_{z}(j+p^{N}\ZZ_{p})= \frac{z^{j}}{1-z^{p^{N}}} \ (0\leq  j \leq p^{N}-1)$$
for $z \in \PP^{1}(\CC_{p})$ with $|z-1|_{p} \geq 1$.

\begin{defn}\label{defn p-adic multiple L-function}
  Let $p$ be a prime number and $(s_i)_r$ be an element of the set
  $$\Xfk_{r}(d) := \{ (s_{1}, \ldots, s_{r}) \in \CC_{p}^{r}\ |\ |s_{j}|_{p} \leq d^{-1}p^{-\frac{1}{p-1}} \ (1 \leq j \leq r)\},$$
  where $d\in \RR_{>0}$. Let $(k_i)_r \in \ZZ^r$, $c \in \NN_{>1}$ which is prime to $p$, and
  $$(\ZZ_{p}^{r})' :=\{ (x_1,\dots, x_r) \in \ZZ_{p}^{r} \ |\ p \nmid x_{1},\ p \nmid (x_{1}+x_{2}),\ \ldots,\ p \nmid (x_{1}+\cdots+x_{r}) \}.$$ 
  The {\it $p$-adic multiple $L$-function} ($p$MLF) is defined by 
  \begin{multline*}
    L_{p,r} ((s_i)_r;(\omega^{k_{i}})_r;(1)_r;c) \coloneqq \\
    \int_{ (\ZZ_{p}^{r})'}\langle x_{1}  \rangle^{-s_{1}}\langle x_{1}   + x_{2}  \rangle^{-s_{2}}\cdots\langle x_{1}  + \cdots + x_{r}  \rangle^{-s_{r}} \\
    \omega^{k_{1}}(x_{1})\omega^{k_{2}}(x_{1}+x_{2})\cdots\omega^{k_{r}}(x_{1} + \cdots + x_{r} )\prod_{i=1}^{r} d\tilde{\mfk}_{c}(x_{i}),
  \end{multline*}
  where $\tilde{\mfk}_{c}$ is the measure defined by
  $$\tilde{\mfk}_{c}:=\sum_{\substack{\xi^{c}=1\\\xi \neq 1}} \mfk_{\xi}.$$
\end{defn}

\begin{defn}\label{defn cyclotomic multiple harmonic value}
  Let $(n_i)_r:=(n_{1},\ldots,n_{r}) \in \NN^r$, and $(\epsilon_i)_r:=(\epsilon_{1},\ldots,\epsilon_{r})\in\mu_{c}^r$, and $m \in \NN$.
  \begin{enumerate}[(i)]
    \item The {\it cyclotomic multiple harmonic sum} (CMHS) is defined to be the element of $\QQ(\mu_{c})$ given by
    $$H_{m}\left((n_{i})_{r};(\epsilon_{i})_{r}\right) \coloneqq \sum_{0<m_{1}<\cdots<m_{r}<m} \frac{( \frac{\epsilon_{2}}{\epsilon_{1}})^{m_{1}} \cdots (\frac{1}{\epsilon_{r}})^{m_{r}}}{m_{1}^{n_{1}} \cdots m_{r}^{n_{r}}}.$$
    \item Let $\Pcal_{c}$ be the set of prime numbers which do not divide $c$. For $p \in \Pcal_{c}$, we also denote by $\epsilon_{i}$ ($1 \leq i \leq r$) the image of $\epsilon_{i}$ by the embedding $\QQ(\mu_{c})\hookrightarrow \QQ_{p}(\mu_{c})$. The {\it cyclotomic multiple harmonic value} (CMHV) is the family of CMHSs defined by
        $$\Hfk\left((n_{i})_{r};(\epsilon_{i})_{r}\right) \coloneqq \left(  p^{n_{1}+\cdots+n_{r}}H_{p}\left((n_{i})_r;(\epsilon_{i})_r\right)\right)_{p \in \Pcal_{c}} \in  \prod_{p \in \Pcal_{c}} \QQ_{p}(\mu_{c}).$$
  \end{enumerate}
\end{defn}

\begin{defn}
  For $(x_p)_p \in \prod_{p \in \Pcal_c} \QQ_p(\mu_c)$, we define $(x_p)_p^{\mathrm{Frob}^{-1}}$ to be $(\mathrm{Frob}_p^{-1}(x_p))_p$, where
  $$\mathrm{Frob}_{p} : \QQ_p(\mu_c) \rightarrow \QQ_p(\mu_c)$$
  is the Frobenius map sending $\xi \mapsto \xi^p$ for all $\xi \in \mu_c$.
\end{defn}

\begin{defn}\label{defn J}
  Let $E_r$ be the set of triples $J = (P_1, P_2, P_3)$ of subsets of $\{1, \ldots, r\}$ such that $1 \in P_2$ and $P_1 \sqcup P_2 \sqcup P_3 = \{1, \ldots, r\}$. Put $(l_{i})_{r}=(l_{1},\ldots,l_{r})\in \NN_{0}^{r}$ and $(\epsilon_{1},\ldots,\epsilon_{r})\in\mu_c^r$. For $j\in P_2 \cup P_3$ we define
  $$l_j^{(P_{2, 3})} = \sum_{k = j}^{\mathrm{min}\{(P_2 \cup P_3) \cap \{j + 1, \ldots, r\}\} - 1} l_k$$
  and
  $$\epsilon_j^{(P_{2, 3})} = \epsilon_{\mathrm{min}\{(P_2 \cup P_3) \cap \{j + 1, \ldots, r\}\}}$$
  if $j \neq \mathrm{max}\{P_2 \cup P_3\}$,
  $$l_j^{(P_{2, 3})} = \sum_{k = j}^{r} l_k$$
  and
  $$\epsilon_j^{(P_{2, 3})} = 1$$
  if $j = \mathrm{max}\{P_2 \cup P_3\}$.
  For $j\in P_3$ we define
  $$j_{(P_{2})} = \mathrm{max}\{\{1, \ldots, j - 1\} \cap P_2\}$$
  and
  $$\kappa'_j = \sharp(P_3 \cap \{j_{(P_2)}, \ldots, j\}).$$
  For $i\in P_2$ with $i \neq \mathrm{max}\{P_2\}$, we define
  $$i^{(P_{2})} = \mathrm{min}\{\{i + 1, \ldots, r\} \cap P_2\},$$
  and we define
  $$\kappa_i = \sharp(P_3 \cap \{i, \ldots, i^{(P_{2})}\})$$
  if $i \neq \mathrm{max}\{P_2\}$ and
  $$\kappa_i = 0$$
  if $i = \mathrm{max}\{P_2\}$.
  For $J \in E_r$, we define
  $$\mathbf{l}_{J} = (l_{J, 1}, \ldots, l_{J, \sharp(P_2)}) \coloneqq \left( l^{(P_{2,3})}_{i}+\sum_{\substack{j\in P_{3}\\ j_{(P_{2})}=i}}l^{(P_{2,3})}_{j}\right)_{i\in P_{2}},$$
  $$\boldsymbol{\epsilon}_{J} = (\epsilon_{J, 1}, \ldots, \epsilon_{J, \sharp(P_2)}) \coloneqq (\epsilon^{-1}_i)_{i\in P_2},$$
  $$\boldsymbol{\kappa}_{J} = (\kappa_{J, 1}, \ldots, \kappa_{J, \sharp(P_2) - 1}) \coloneqq (\kappa_{i})_{i \in P_{2}\setminus \{\mathrm{max}\{P_2\}\}}.$$
\end{defn}

\begin{defn}
  Let $J\in E_r$ be the triple $(P_1, P_2, P_3)$. We define
  $$T_{r, J} \coloneqq \left\{ (t_i)_r \in [1, p-1]^r \left| \ \substack{ t_{i-1} \leq t_i, \text{ if } i \in P_1 \\ t_{i-1} > t_i, \text{ if } i \in P_3} \right. \right\}.$$
\end{defn}

\begin{defn}
  Let $A = (A_1, \ldots, A_k)$ be a sequence of subsets of $\{1, \ldots, r\}$ such that the $A_i$ are pairwise disjoint. The quasi-simplex $\delta_{r, A}$ is defined by
  $$\delta_{r, A} \coloneqq \left\{ (t_i)_r \in [1, p-1]^r \left| \substack{ t_a = t_{a'} \text{ if } a \in A_i, a' \in A_{i'}, i = i' \\ t_a < t_{a'} \text{ if } a' \in A_i, a \in A_{i'}, i < i' } \right. \right\}.$$
\end{defn}

\begin{defn}
  Let $S$ be a subset of $\{1, \ldots, r\}$. The set $\Delta_{r, S}$ of quasi-simplices is defined by
  $$\Delta_{r, S} \coloneqq \left\{ \delta_{r, A} \left| \bigcup_{i = 1}^{k} A_i = S \right. \right\}.$$
\end{defn}

\begin{defn}
  Let $\delta = \delta_{r, A}$ be a quasi-simplex. For an index $\bm{l} = (l_i)_r \in \NN_0^r$, and the pair $\bm{w} = ((n_i)_r, (\epsilon_i)_r) \in \NN^r \times \mu_c^r$, we define $\bm{w}(\bm{l})_\delta \coloneqq (\bm{n}(\bm{l})_\delta, \epsilon(\bm{l})_\delta)$, where
  $$\bm{n}(\bm{l})_\delta = \left( \sum_{a \in A_i} (n_a + l_a) \right)_{i \in \{1, \ldots, k\}} \text{ and } \epsilon(\bm{l})_\delta = \left( \prod_{j = i}^{k} \prod_{a \in A_j} \frac{\epsilon_{a + 1}}{\epsilon_{a}} \right)_{i \in \{1, \ldots, k\}}$$
  with $\epsilon_{r + 1} = 1$.
\end{defn}

\begin{defn}
  Let $X$ be a subset of $\NN$. We define $X' \coloneqq X \cup (X - 1)$, where $X - 1 \coloneqq \{x - 1 \in \NN | x \in X\}$.
\end{defn}

\begin{prop}[{\cite[Proposition 5]{CMBN}}]
  Let $J\in E_r$ be the triple $(P_1, P_2, P_3)$. Then we have a unique decomposition
  $$T_{r, J} = \bigsqcup_{\delta \in \Delta} \delta$$
  such that $\Delta \subset \Delta_{r, P_1' \cup P_3'}$.
\end{prop}

\begin{defn}
  Let $\mathrm{proj}_{S} : \Delta_{r, \{1, \ldots, r\}} \rightarrow \Delta_{r, S}$ be the projection defined by $\mathrm{proj}_{S}(\delta_{r, A}) = \delta_{r, (A_i \cap S)_k}$. We define the set of quasi-simplices $\Delta(T_{r, J})$ to be the preimage $\mathrm{proj}_{P_1' \cup P_3'}(\Delta)$, where $\Delta$ is the set of quasi-simplices in the decomposition of the previous proposition.
\end{defn}

Using the notation above, Furusho and Jarossay \cite{CMBN} proved the following theorem.

\begin{thm}[{\cite[Theorem 1]{CMBN}}]\label{thm CMBN}
  For any $r$-tuple $(n_i)_r$ of positive integers, the family
  $$\left( p^{\sum_{i=1}^{r}n_{i}}L_{p,r}((n_{i})_{r}; (\omega^{-n_{i}})_{r}; (1)_{r}; c)\right)_{p \in \Pcal_{c}}$$ 
  is expressed as
  \begin{multline*}
    \sum_{{\bf l}=(l_{i})_{r} \in \mathbb{N}_{0}^{r}} \sum_{\substack{ \epsilon \in \mu_{c} \\ (\epsilon_{i})_{r} \in (\mu_{c} \setminus \{1\})^{r}}} \sum_{J=(P_{1},P_{2},P_{3}) \in E_{r}} \sum_{\substack{ \xi \in \mu_{c} \\ \prod_{j \in P_{3}} \left(\frac{\epsilon_{j}^{(P_{2,3})}}{\epsilon_{j}}\right)^{-\kappa'_{j}} \xi = \epsilon}} \\
    \frac{\epsilon}{\prod_{i=1}^{r}(1 - \epsilon_{i})}  \prod_{i=1}^{r} {-n_{i} \choose l_{i}} \Bcal_{0,\xi}^{(\mathbf{l}_{J},\boldsymbol{\epsilon}_{J},\boldsymbol{\kappa}_{J})} \sum_{\delta \in \Delta(T_{r,J})}  \Hfk({\bf w}({\bf l})_{\delta})^{\Frob^{-1}}.
  \end{multline*}
  This is an infinite series whose terms are $\QQ(\mu_{c})$-linear combinations of CMHVs of depth at most $r$ and of weights tending to infinity. It converges in $\prod_{p\in \Pcal_{c}} \QQ_{p}(\mu_{c})$ with respect to the topology of uniform convergence in $p \in \Pcal_{c}$.
\end{thm}

Our main theorem is the following reformulation of the above expansion, in which the coefficients are expressed in terms of the special values of the multiple zeta-function.

\begin{thm}\label{thm Main thm}
  For any $r$-tuple $(n_i)_r$ of positive integers, the family
  $$\left( p^{\sum_{i=1}^{r}n_{i}}L_{p,r}((n_{i})_{r}; (\omega^{-n_{i}})_{r}; (1)_{r}; c)\right)_{p \in \Pcal_{c}}$$
  can be expressed as
  \begin{multline*}
    \sum_{{\bf l}=(l_{i})_{r} \in \NN_{0}^{r}} \sum_{\substack{ \epsilon \in \mu_{c} \\ (\epsilon_{i})_{r} \in (\mu_{c} \setminus \{1\})^{r}}} \sum_{J=(P_{1},P_{2},P_{3}) \in E_{r}} \\
    \frac{\epsilon}{\prod_{i=1}^{r}(1 - \epsilon_{i})}  \prod_{i=1}^{r} {-n_{i} \choose l_{i}} \sum_{(e_i)_{\sharp(P_2)} \in \{0\}\times \{0, 1\}^{\sharp(P_2)-1}} \sum_{i_1 = 0}^{1-e_1} \cdots \sum_{i_{\sharp(P_2)} = 0}^{1-e_{\sharp(P_2)}} (-1)^{\sharp(P_2) + \sum_{k = 1}^{\sharp(P_2)} (i_k + l_{J, k})}\\
    \zeta_{\sharp(P_2)-\sum_{k = 1}^{\sharp(P_2)}e_k}\left( (-l_{J, 1} \circ_{e_2} \cdots \circ_{e_{\sharp(P_2)}} -l_{J, \sharp(P_2)}); (\epsilon_{J, 1} (1 - e_1) \circ_{e_2} \cdots \circ_{e_{\sharp(P_2)}} \epsilon_{J, \sharp(P_2)} (1 - e_{\sharp(P_2)})) \right) \\
    \delta_{ \epsilon \prod_{j \in P_{3}} \left(\frac{\epsilon_{j}^{(P_{2,3})}}{\epsilon_{j}}\right)^{\kappa'_{j}}, \prod_{k = 1}^{\sharp(P_2)} \epsilon_{J, k}^{-(\kappa_{J, k-1} + i_k)}} \sum_{\delta \in \Delta(T_{r,J})}  \Hfk({\bf w}({\bf l})_{\delta})^{\Frob^{-1}}.
  \end{multline*}
\end{thm}

\begin{proof}
  Substituting Corollary \ref{cor l = 0 formula} into Theorem \ref{thm CMBN}, we obtain
  \begin{multline*}
    \sum_{{\bf l}=(l_{i})_{r} \in \mathbb{N}_{0}^{r}} \sum_{\substack{ \epsilon \in \mu_{c} \\ (\epsilon_{i})_{r} \in (\mu_{c} \setminus \{1\})^{r}}} \sum_{J=(P_{1},P_{2},P_{3}) \in E_{r}} \sum_{\substack{ \xi \in \mu_{c} \\ \prod_{j \in P_{3}} \left(\frac{\epsilon_{j}^{(P_{2,3})}}{\epsilon_{j}}\right)^{-\kappa'_{j}} \xi = \epsilon}} \\
    \frac{\epsilon}{\prod_{i=1}^{r}(1 - \epsilon_{i})}  \prod_{i=1}^{r} {-n_{i} \choose l_{i}} \sum_{(e_i)_{\sharp(P_2)} \in \{0\}\times \{0, 1\}^{\sharp(P_2)-1}} \sum_{i_1 = 0}^{1-e_1} \cdots \sum_{i_{\sharp(P_2)} = 0}^{1-e_{\sharp(P_2)}} (-1)^{\sum_{k = 1}^{\sharp(P_2)} (e_k + i_k)}\\
     \Bfk_{\left(l_{J, 1} \circ_{e_2} \cdots \circ_{e_{\sharp(P_2)}} l_{J, \sharp(P_2)}\right)} (\epsilon_{J, 1}^{-1} (1 - e_1) \circ_{e_2} \cdots \circ_{e_{\sharp(P_2)}} \epsilon_{J, \sharp(P_2)}^{-1} (1 - e_{\sharp(P_2)})) \\
     \delta_{\xi, \prod_{k = 1}^{\sharp(P_2)} \epsilon_{J, k}^{-(\kappa_{J, k-1} + i_k)}} \sum_{\delta \in \Delta(T_{r,J})}  \Hfk({\bf w}({\bf l})_{\delta})^{\Frob^{-1}}.
  \end{multline*}
   Applying Theorem \ref{thm TMBN} then completes the proof.
\end{proof}

The following corollary gives the explicit formula in the case of $r=1$.
\begin{cor}\label{cor depth 1}
  For any $n \in \NN$, the family
  $$\left( p^{n}L_{p,1}(n; \omega^{-n}; 1; c)\right)_{p \in \Pcal_{c}}$$
  is expressed as
  $$\sum_{l = 0}^{\infty} \sum_{\epsilon \in \mu_{c} \setminus \{1\}} {-n \choose l} (-1)^{1 + l} \zeta_{1}\left(-l; \epsilon^{-1}; 1\right) \Hfk(n + l, \epsilon^{-1})^{\Frob^{-1}}.$$
\end{cor}

\begin{proof}
  Substituting $r = 1$ into Theorem \ref{thm Main thm}, we get
  \begin{multline*}
    \sum_{l_1 = 0}^{\infty} \sum_{\substack{ \epsilon \in \mu_{c} \\ \epsilon_1 \in \mu_{c} \setminus \{1\}}} \sum_{J\in E_1=\{(\emptyset,\{1\},\emptyset)\}} \\
    \frac{\epsilon}{(1 - \epsilon_{1})}  {-n_{1} \choose l_{1}} \sum_{(e_i)_{\sharp(P_2)} \in \{0\}\times \{0, 1\}^{\sharp(P_2)-1}} \sum_{i_1 = 0}^{1-e_1} \cdots \sum_{i_{\sharp(P_2)} = 0}^{1-e_{\sharp(P_2)}} (-1)^{\sharp(P_2) + \sum_{k = 1}^{\sharp(P_2)} (i_k + l_{J, k})}\\
    \zeta_{\sharp(P_2)-\sum_{k = 1}^{\sharp(P_2)}e_k}\left( (-l_{J, 1} \circ_{e_2} \cdots \circ_{e_{\sharp(P_2)}} -l_{J, \sharp(P_2)}); (\epsilon_{J, 1} (1 - e_1) \circ_{e_2} \cdots \circ_{e_{\sharp(P_2)}} \epsilon_{J, \sharp(P_2)} (1 - e_{\sharp(P_2)})) \right) \\
    \delta_{\epsilon \prod_{j \in P_{3}} \left(\frac{\epsilon_{j}^{(P_{2,3})}}{\epsilon_{j}}\right)^{\kappa'_{j}}, \prod_{k = 1}^{\sharp(P_2)} \epsilon_{J, k}^{-(\kappa_{J, k-1} + i_k)}} \sum_{\delta \in \Delta(T_{1,J})}  \Hfk({\bf w}(l_1)_{\delta})^{\Frob^{-1}}.
  \end{multline*}
  Note that $J = (\emptyset,\{1\},\emptyset)$ and $e_{\sharp(P_2)} = e_1 = 0$. Then we have
  $$\sum_{l_1 = 0}^{\infty} \sum_{\substack{ \epsilon \in \mu_{c} \\ \epsilon_1 \in \mu_{c} \setminus \{1\}}} \frac{\epsilon}{(1 - \epsilon_{1})} {-n_{1} \choose l_{1}} \sum_{i_1 = 0}^{1} (-1)^{1 + i_1 + l_{1}} \zeta_{1}\left(-l_{1}; \epsilon_1^{-1}; 1\right) \delta_{\epsilon, \epsilon_{1}^{i_1}} \sum_{\delta \in \Delta(T_{1,(\emptyset,\{1\},\emptyset)})}  \Hfk({\bf w}(l_1)_{\delta})^{\Frob^{-1}}.$$
  When $i_1 = 0$, we obtain
  $$\sum_{l_1 = 0}^{\infty} \sum_{\epsilon_1 \in \mu_{c} \setminus \{1\}} \frac{1}{(1 - \epsilon_{1})} {-n_{1} \choose l_{1}} (-1)^{1 + l_{1}} \zeta_{1}\left(-l_{1}; \epsilon_1^{-1}; 1\right) \sum_{\delta \in \Delta(T_{1,(\emptyset,\{1\},\emptyset)})}  \Hfk({\bf w}(l_1)_{\delta})^{\Frob^{-1}}.$$
  When $i_1 = 1$, we obtain
  $$\sum_{l_1 = 0}^{\infty} \sum_{\epsilon_1 \in \mu_{c} \setminus \{1\}} \frac{-\epsilon_{1}}{(1 - \epsilon_{1})} {-n_{1} \choose l_{1}} (-1)^{1 + l_{1}} \zeta_{1}\left(-l_{1}; \epsilon_1^{-1}; 1\right) \sum_{\delta \in \Delta(T_{1,(\emptyset,\{1\},\emptyset)})}  \Hfk({\bf w}(l_1)_{\delta})^{\Frob^{-1}}.$$
  Combining the two terms above, we obtain
  $$\sum_{l_1 = 0}^{\infty} \sum_{\epsilon_1 \in \mu_{c} \setminus \{1\}} {-n_{1} \choose l_{1}} (-1)^{1 + l_{1}} \zeta_{1}\left(-l_{1}; \epsilon_1^{-1}; 1\right) \sum_{\delta \in \Delta(T_{1,(\emptyset,\{1\},\emptyset)})}  \Hfk({\bf w}(l_1)_{\delta})^{\Frob^{-1}}.$$
  Since
  $$\sum_{\delta \in \Delta(T_{1,(\emptyset,\{1\},\emptyset)})}  \Hfk({\bf w}(l_1)_{\delta})^{\Frob^{-1}} = \Hfk(n_1 + l_1, \epsilon_1^{-1})^{\Frob^{-1}},$$
  the corollary follows.
\end{proof}

The following corollary gives the explicit formula in the case of $r=2$.
\begin{cor}\label{cor depth 2}
  For any $(n_1, n_2) \in \NN^2$, the family
  $$\left( p^{n_1 + n_2}L_{p,2}((n_1, n_2); (\omega^{-n_{1}}, \omega^{-n_{2}}); (1, 1); c)\right)_{p \in \Pcal_{c}}$$
  can be expressed as
  \begin{multline*}
    \sum_{(l_1, l_2) \in \mathbb{N}_{0}^{2}} \sum_{(\epsilon_{1}, \epsilon_{2}) \in (\mu_{c} \setminus \{1\})^{2}} {-n_{1} \choose l_{1}} {-n_{2} \choose l_{2}} (-1)^{l_{1} + l_{2}} \\
    \left\{ \zeta_{2}\left( (-l_{1}, -l_{2}); (\epsilon_1^{-1}, \epsilon_2^{-1}) \right) \left( \Hfk((n_1 + l_1, n_2 + l_2), (\epsilon_1^{-1}, \epsilon_2^{-1}))^{\Frob^{-1}} + \Hfk(n_1 + l_1 + n_2 + l_2, \epsilon_1^{-1})^{\Frob^{-1}} \right) \right. \\
    + \left. \left( \zeta_{2}\left( (-l_{1}, -l_{2}); (\epsilon_1^{-1}, \epsilon_2^{-1}) \right) + \zeta_{1}\left( - l_{1} - l_{2}; \epsilon_1^{-1} \right) \right) \Hfk((n_2 + l_2, n_1 + l_1), (\epsilon_1^{-1}, \epsilon_1^{-1} \epsilon_2))^{\Frob^{-1}} \right\}. \\
  \end{multline*}
\end{cor}

\begin{proof}
  Substituting $r = 2$ into Theorem \ref{thm Main thm}, we get
  \begin{multline*}
    \sum_{(l_1, l_2) \in \mathbb{N}_{0}^{2}} \sum_{\substack{ \epsilon \in \mu_{c} \\ (\epsilon_{1}, \epsilon_{2}) \in (\mu_{c} \setminus \{1\})^{2}}} \sum_{J=(P_{1},P_{2},P_{3}) \in \{(\{2\}, \{1\}, \emptyset), (\emptyset, \{1, 2\}, \emptyset), (\emptyset, \{1\}, \{2\})\}} \\
    \frac{\epsilon}{(1 - \epsilon_{1})(1 - \epsilon_{2})} {-n_{1} \choose l_{1}} {-n_{2} \choose l_{2}} \sum_{(e_i)_{\sharp(P_2)} \in \{0\}\times \{0, 1\}^{\sharp(P_2)-1}} \sum_{i_1 = 0}^{1-e_1} \cdots \sum_{i_{\sharp(P_2)} = 0}^{1-e_{\sharp(P_2)}} (-1)^{\sharp(P_2) + \sum_{k = 1}^{\sharp(P_2)} (i_k + l_{J, k})}\\
    \zeta_{\sharp(P_2)-\sum_{k = 1}^{\sharp(P_2)}e_k}\left( (-l_{J, 1} \circ_{e_2} \cdots \circ_{e_{\sharp(P_2)}} -l_{J, \sharp(P_2)}); (\epsilon_{J, 1} (1 - e_1) \circ_{e_2} \cdots \circ_{e_{\sharp(P_2)}} \epsilon_{J, \sharp(P_2)} (1 - e_{\sharp(P_2)})) \right) \\
    \delta_{ \epsilon \prod_{j \in P_{3}} \left(\frac{\epsilon_{j}^{(P_{2,3})}}{\epsilon_{j}}\right)^{\kappa'_{j}}, \prod_{k = 1}^{\sharp(P_2)} \epsilon_{J, k}^{-(\kappa_{J, k-1} + i_k)}} \sum_{\delta \in \Delta(T_{2,J})}  \Hfk({\bf w}({\bf l})_{\delta})^{\Frob^{-1}}.
  \end{multline*}
  We consider the three possible cases for $J$ as follows:
  \begin{enumerate}[\text{Case }1:]
    \item $J = (\{2\}, \{1\}, \emptyset)$: by Definition \ref{defn J}, we have
    $$l_{J, 1} = l_1 + l_2 \text{ and } \epsilon_{J, 1} = \epsilon_{1}^{-1}.$$
    In this case $\sharp(P_2) = 1$. Using the same calculation as in the proof of Corollary \ref{cor depth 1}, we obtain
    \begin{multline*}
      \sum_{(l_1, l_2) \in \mathbb{N}_{0}^{2}} \sum_{(\epsilon_{1}, \epsilon_{2}) \in (\mu_{c} \setminus \{1\})^{2}} \\
      \frac{1}{1 - \epsilon_{2}} {-n_{1} \choose l_{1}} {-n_{2} \choose l_{2}} (-1)^{1 + l_{1} + l_{2}} \zeta_{1}\left(-(l_{1} + l_{2}); \epsilon_1^{-1} \right)\\
      \sum_{\delta \in \Delta(T_{2, (\{2\}, \{1\}, \emptyset)})} \Hfk({\bf w}({\bf l})_{\delta})^{\Frob^{-1}}.
    \end{multline*}
    Moreover,
    $$\sum_{\delta \in \Delta(T_{2, (\{2\}, \{1\}, \emptyset)})} \Hfk({\bf w}({\bf l})_{\delta})^{\Frob^{-1}} = \Hfk((n_1 + l_1, n_2 + l_2), (\epsilon_1^{-1}, \epsilon_2^{-1}))^{\Frob^{-1}} + \Hfk(n_1 + l_1 + n_2 + l_2, \epsilon_1^{-1})^{\Frob^{-1}}.$$
    \item $J = (\emptyset, \{1, 2\}, \emptyset)$: by Definition \ref{defn J}, we have
    $$l_{J, 1} = l_1, \ l_{J, 2} = l_2, \ \epsilon_{J, 1} = \epsilon_{1}^{-1}, \ \epsilon_{J, 2} = \epsilon_{2}^{-1}, \text{ and } \kappa_1 = 0.$$
    In this case $\sharp(P_2) = 2$. Using the same calculation as in the proof of Corollary \ref{cor depth 1}, we obtain
    \begin{multline*}
      \sum_{(l_1, l_2) \in \mathbb{N}_{0}^{2}} \sum_{(\epsilon_{1}, \epsilon_{2}) \in (\mu_{c} \setminus \{1\})^{2}} \\
      {-n_{1} \choose l_{1}} {-n_{2} \choose l_{2}} (-1)^{l_{1} + l_{2}} \zeta_{2}\left( (-l_{1}, -l_{2}); (\epsilon_1^{-1}, \epsilon_2^{-1}) \right)\\
      \sum_{\delta \in \Delta(T_{2, (\emptyset, \{1, 2\}, \emptyset)})} \Hfk({\bf w}({\bf l})_{\delta})^{\Frob^{-1}}
    \end{multline*}
    when $e_2 = 0$, and
    \begin{multline*}
      \sum_{(l_1, l_2) \in \mathbb{N}_{0}^{2}} \sum_{(\epsilon_{1}, \epsilon_{2}) \in (\mu_{c} \setminus \{1\})^{2}} \\
      {-n_{1} \choose l_{1}} {-n_{2} \choose l_{2}} (-1)^{l_{1} + l_{2}} \frac{1}{1 - \epsilon_{2}} \zeta_{1}\left( - l_{1} - l_{2}; \epsilon_1^{-1} \right)\\
      \sum_{\delta \in \Delta(T_{2, (\emptyset, \{1, 2\}, \emptyset)})} \Hfk({\bf w}({\bf l})_{\delta})^{\Frob^{-1}}.
    \end{multline*}
    when $e_2 = 1$.
    We have
    \begin{align*}
      \sum_{\delta \in \Delta(T_{2, (\emptyset, \{1, 2\}, \emptyset)})} \Hfk({\bf w}({\bf l})_{\delta})^{\Frob^{-1}} = & \Hfk((n_1 + l_1, n_2 + l_2), (\epsilon_1^{-1}, \epsilon_2^{-1}))^{\Frob^{-1}} \\
      & + \Hfk(n_1 + l_1 + n_2 + l_2, \epsilon_1^{-1})^{\Frob^{-1}} \\
      & + \Hfk((n_2 + l_2, n_1 + l_1), (\epsilon_1^{-1}, \epsilon_1^{-1} \epsilon_2))^{\Frob^{-1}}.
    \end{align*}
    \item $J = (\emptyset, \{1\}, \{2\})$: by Definition \ref{defn J}, we have
    $$l_{J, 1} = l_1 + l_2, \ \epsilon_{J, 1} = \epsilon_{1}^{-1} \text{ and } \kappa'_2 = 1.$$
    In this case $\sharp(P_2) = 1$. Using the same calculation as in the proof of Corollary \ref{cor depth 1}, we obtain
    \begin{multline*}
      \sum_{(l_1, l_2) \in \mathbb{N}_{0}^{2}} \sum_{(\epsilon_{1}, \epsilon_{2}) \in (\mu_{c} \setminus \{1\})^{2}}\\
      \frac{\epsilon_{2}}{1 - \epsilon_{2}} {-n_{1} \choose l_{1}} {-n_{2} \choose l_{2}} (-1)^{1 + l_{1} + l_{2}} \zeta_{1}\left( - l_{1} - l_{2}; \epsilon_1^{-1} \right) \\
      \sum_{\delta \in \Delta(T_{2, (\emptyset, \{1\}, \{2\})})}  \Hfk({\bf w}({\bf l})_{\delta})^{\Frob^{-1}}.
    \end{multline*}
    Moreover,
    $$\sum_{\delta \in \Delta(T_{2, (\emptyset, \{1\}, \{2\})})}  \Hfk({\bf w}({\bf l})_{\delta})^{\Frob^{-1}} = \Hfk((n_2 + l_2, n_1 + l_1), (\epsilon_1^{-1}, \epsilon_1^{-1} \epsilon_2))^{\Frob^{-1}}.$$
  \end{enumerate}
  Combining all the terms above, we obtain
  \begin{multline*}
    \sum_{(l_1, l_2) \in \mathbb{N}_{0}^{2}} \sum_{(\epsilon_{1}, \epsilon_{2}) \in (\mu_{c} \setminus \{1\})^{2}} {-n_{1} \choose l_{1}} {-n_{2} \choose l_{2}} (-1)^{l_{1} + l_{2}} \\
    \left\{ \zeta_{2}\left( (-l_{1}, -l_{2}); (\epsilon_1^{-1}, \epsilon_2^{-1}) \right) \left( \Hfk((n_1 + l_1, n_2 + l_2), (\epsilon_1^{-1}, \epsilon_2^{-1}))^{\Frob^{-1}} + \Hfk(n_1 + l_1 + n_2 + l_2, \epsilon_1^{-1})^{\Frob^{-1}} \right) \right. \\
    + \left. \left( \zeta_{2}\left( (-l_{1}, -l_{2}); (\epsilon_1^{-1}, \epsilon_2^{-1}) \right) + \zeta_{1}\left( - l_{1} - l_{2}; \epsilon_1^{-1} \right) \right) \Hfk((n_2 + l_2, n_1 + l_1), (\epsilon_1^{-1}, \epsilon_1^{-1} \epsilon_2))^{\Frob^{-1}} \right\}.
  \end{multline*}
  This proves the corollary.
\end{proof}

\end{document}